\documentclass[12pt,twoside]{amsart}

\usepackage{graphicx}

\usepackage{bbm}
\usepackage{amsmath,amsthm,amssymb,amscd}
\usepackage{enumerate}
\usepackage{latexsym}
\usepackage[top=1.5in, bottom=1.5in, left=1.25in, right=1.25in]	 {geometry}

\usepackage{tikz}
\usetikzlibrary{arrows, decorations.pathmorphing, backgrounds, positioning, fit, petri}
\usetikzlibrary{ decorations.text}
\usepackage{multicol}

\def\({\left(}
\def \){ \right)}
\def\[{\left[}
\def \]{ \right]}

 \def\N{{\mathcal N}}

\usepackage{tikz}\usepackage{tikz-3dplot}

\usepackage{mathtools}
\mathtoolsset{showonlyrefs,showmanualtags}

\usepackage{hyperref} %,hypertexnames=false,colorlinks,[pagebackref]
\hypersetup{
    colorlinks=true,       % false: boxed links; true: colored links
    linkcolor=blue,          % color of internal links
    citecolor=magenta,        % color of links to bibliography
    filecolor=magenta,      % color of file links
    urlcolor=cyan           % color of external links
}

\theoremstyle{plain} % definition
\newtheorem{theorem}{Theorem}[section]
\newtheorem{lemma}[theorem]{Lemma}
\newtheorem{proposition}[theorem]{Proposition}

\theoremstyle{definition}
\newtheorem{definition}[equation]{Definition}

\theoremstyle{remark}
\newtheorem{remark}[equation]{Remark}

\numberwithin{equation}{section}

\begin{document}

\title[Hyperuniform  point sets on flat tori: deterministic and probabilistic ]
{Hyperuniform  point sets on flat tori: deterministic and probabilistic aspects}

\author{Tetiana A. Stepanyuk}
\address{Graz University of Technology, Institute of Analysis and Number Theory, Kopernikusgasse 24/II 8010, Graz, Austria}
\address{Present address: Johann Radon Institute for Computational and Applied Mathematics (RICAM)
Austrian Academy of Sciences, Altenbergerstrasse 69 4040, Linz, Austria;
   Institute of Mathematics of NAS of Ukraine, 3, Tereshchenkivska st., 01601, Kyiv-4, Ukraine}

\email{tania$_{-}$stepaniuk@ukr.net}

\thanks{The author is supported by the Austrian Science Fund FWF projects F5503 and F5506-N26 (part of the Special Research Program (SFB) ``Quasi-Monte Carlo Methods: Theory and Applications'') }

\maketitle

\begin{abstract}
In this paper 
we study hyperuniformity on flat tori. Hyperuniform point sets on the unit sphere have been studied by J.~Brauchart, P.~Grabner, W.~Kusner and J.~Ziefle in \cite{BrauchartGrabnerKusner2018} and \cite{BrauchartGrabnerKusnerZiefle2018}. It is shown that point sets which are hyperuniform for large balls, small balls or balls of threshold order on the flat tori are uniformly distributed. Moreover, it  is also shown that QMC--designs sequences for Sobolev classes,  probabilistic point sets (with respect to jittered samplings) and some determinantal point process are hyperuniform.
\end{abstract}



\section{Introduction}\label{intro}

The concept of hyperuniformity had been introduced by S. Torquato and F. Stillinger \cite{TorquatoStillinger} to measure regularity of distributions of infinite particle systems in $\mathbb{R}^{d}$.
A hyperuniform many--particle system  \cite{Torquato2018} in $d$--dimensional Euclidean space $\mathbb{R}^{d}$ is one in which normalized density fluctuations are completely suppressed at very large length scales, implying that the structure factor $S(\mathbf{k})$ tends to zero in the limit $|\mathbf{k}|\rightarrow0$. Equivalently, a hyperuniform system is one in which the number variance
of particles within a spherical observation window of radius $R$ grows more slowly than the window volume in the large $R$--limit, i.e., slower than $R^{d}$.

Hyperuniformity found a number of different applications in physics and beyond physics (see e.g., \cite{JiaoLauHatzikirouMeyer-HermannCorboTorquato} 
\cite{OguzSocolarSteinhardtTorquato} 
 \cite{Torquato2018}).
For example, it was observed, that all perfect crystals, perfect quasicrystals and special disordered systems are hyperuniform.
So, the hyperuniformity concept enables a unified framework to classify and structurally characterize crystals, quasicrystals and the exotic disordered varieties.

In  \cite{BrauchartGrabnerKusner2018} and \cite{BrauchartGrabnerKusnerZiefle2018} three regimes of hyperuniformity for sequences of point sets and for samples of points processes  on the unit sphere were introduced and studied. 
The aim of present paper is to study hyperuniformity of deterministic and probabilistic point sets on flat tori  in three regimes of hyperuniformity.

 Let $\Lambda=A\mathbb{Z}^{d}$ be a lattice in $\mathbb{R}^{d}$, generated by some nonsingular square matrix $A=[v_{1},...,v_{d}]$. Then we identify by
  \begin{align*}
 \Omega_{\Lambda}:=\Big\{w: \ w=\sum\limits_{j=1}^{d}\alpha_{j}v_{j}, \ \ \alpha_{j}\in[0,1), \ j=1,2,...,d \Big\}
 \end{align*}
 the fundamental domain of the quotient space $\mathbb{R}^{d}/ \Lambda$.  The volume of $\Omega_{\Lambda}$, denoted by $|\Lambda|$  equals  $|\mathrm{det} A|$ , and is called the co--volume of $\Lambda$.
 And let $\Lambda^{*}$ be the dual lattice to $\Lambda$, that is, 
 $\Lambda^{*}=\left\{ w\in \mathbb{R}^{d}: \ \langle w,v\rangle \in \mathbb{Z}  \ \mathrm{for \ all} \ v\in\Lambda \right\}=(A^{t})^{-1}\mathbb{Z}^{d}$ .
 
  Let $\Delta$ be the Laplace--Beltrami operator on $\mathbb{R}^{d}/ \Lambda$, which has the sequence of eigenvalues $\left(-4\pi^{2}\langle w,w \rangle\right)_{w\in\Lambda^{*}}$ and a  complete orthornormal system of  eigenfunctions $f_{w}(u)=e^{2\pi\langle u, w\rangle}$, $w\in\Lambda^{*}$, such that
   \begin{equation}
\Delta f_{w}+4\pi^{2}\langle w,w \rangle f_{w}=0
\end{equation}
and
   \begin{equation}
\int\limits_{\Omega_{\Lambda}}f_{w}(u)\overline{f_{w'}(u)}d\mu(u)=\delta_{w,w'}, \ \ w,w'\in \Lambda^{*},
\end{equation}
where $\mu$ is the normalized Lebesgue measure in  $\Omega_{\Lambda}$.
 
 Let $B(\mathbf{x}, R)$ denotes an Euclidean ball of radius $R$ and with center $\mathbf{x}$, $\mathbf{x}\in\mathbb{R}^{d}$.
 The $d$-dimensional volume of the ball $B(\mathbf{x}, R)$  equals
 \begin{equation}\label{volumeBall}
 \mathrm{Vol}(B(\mathbf{x}, R))=\frac{\pi^{\frac{d}{2}}}{\Gamma(\frac{d}{2}+1)}R^{d}.
\end{equation}

The paper is organized  as follows.

 Section \ref{prelim} contains  the necessary background for Bessel functions, definition of hyperuniformity for all three regimes and a computable expression for the number variance.

In Section \ref{Hyperuniform} we prove that sequences of  point sets which are hyperuniform for large balls, small balls or balls of threshold order are uniformly distributed.

Section \ref{QMCHyperuniform} gives the definition of QMC design sequences for Sobolev classes $W^{\alpha,2}(\Omega_{\Lambda})$. Here we prove that  QMC design sequences are hyperuniform in all three regimes.

In Section \ref{JitteredSamplings} we consider the jittered sampling process on flat tori and show that it is hyperuniform in all three regimes.

In Section \ref{determinantalProcess} we study hyperuniformity of random points on flat tori drawn from certain translation invariant determinantal processes. The expected Riesz energy of these determinantal processes on flat tori was computed in \cite{MarzoOrtegsCerda}.
Riesz energy on flat tori was also studied in \cite{HardinSaffSimanekSu2018} 
and
\cite{HardinSaffSimanek}. This process turns out to be hyperuniform for large and small balls and has a slightly weaker behavior in the threshold order regime.

\section{Preliminaries}\label{prelim}

\subsection{Bessel functions}\label{Bessel}
Bessel functions are solutions of  Bessel's differential equation
\begin{align*}
z^{2}\frac{d^{2}x}{dz^{2}}+z\frac{dx}{dz}+(z^{2}-u^{2})x
=0,
\end{align*}
where $u$ and $z$ can be arbitrarily complex.

For small arguments $0<z\ll \sqrt{v+1}$, $v>0$ the following asymptotic formula holds (see, e.g.,  \cite{Abramowitz Stegun}  (9.1.7) and  (9.1.10))
\begin{align}\label{asympSmall}
J_{v}(z)\sim\frac{1}{\Gamma(v+1)}\Big(\frac{z}{2}\Big)^{v}.
\end{align}

For large arguments $z$ we will use the formula (see, e.g.,  \cite{Abramowitz Stegun}  (9.2.1))
\begin{align}\label{asympLarge}
J_{v}(z)=\sqrt{\frac{2}{\pi z}}\Big(\cos\Big( z-\frac{v\pi}{2}-\frac{\pi}{4}\Big)+\mathcal{O}\Big(\frac{1}{|z|} \Big)\Big), \  -\pi<\mathrm{arg} z<\pi .
\end{align}

For Bessel function the following integral representation (see, e.g. \cite{Magnus-Oberhettinger-Soni1966:formulas_theorems} (3.6.2)) holds
\begin{equation}\label{formulaBessel1}
\Gamma\Big(\frac{1}{2}+v\Big)J_{v}(z)=2\pi^{-\frac{1}{2}}\Big( \frac{1}{2}z\Big)^{v}
\int\limits_{0}^{\frac{\pi}{2}}\cos(z\cos t)\sin^{2v}t dt,  \ \ \mathrm{Re}>-\frac{1}{2}.
\end{equation}
To calculate the integral involving Bessel function we will  use the formula (see, e.g. \cite{Magnus-Oberhettinger-Soni1966:formulas_theorems} (3.8.1))
\begin{align}\label{formulaBessel2}
\int z^{v+1}J_{v}(z)dz=z^{v+1}J_{v+1}(z).
\end{align}

\subsection{Hyperuniformity on the flat tori}\label{hyperuniform} 

As in \cite{BrauchartGrabnerKusner2018} we will consider a sequence of finite point sets  $(X_{N})_{N\in A}$, $A\subseteq\mathbb{N}$, assuming that $\# X_{N}=N$. Also, the set $X_{N}=\{\mathbf{x}_{1}^{(N)}, ... \mathbf{x}_{N}^{(N)}\}$ consist of points depending on $N$, but we will omit this dependence for the ease of notation.
 
 \begin{definition}\label{def1}{(Uniform distribution)}
  A sequence of point sets $(X_{N})_{N\in A}$ is called uniformly distributed on 
  $\Omega_{\Lambda}$, if for all balls $B(\mathbf{x}, R)$, $\mathbf{x}\in\mathbb{R}^{d}$, $R\in[0,\frac{1}{2}\mathrm{diam}\Omega_{\Lambda}]$ 
  the relation
\begin{equation}\label{defUniformDistr}
 \lim\limits_{N\rightarrow\infty}\frac{1}{N}\sum\limits_{j=1}^{N}\mathbbm{1}_{B(\mathbf{x}, R)}(\mathbf{x}_{j})=\mathrm{Vol}(B(\mathbf{x}, R))
\end{equation}
holds.
\end{definition}

It follows from the Weyl criterion (see, e.g., the book about the general theory of uniform distributions \cite{KuipersandNiederreiter}),
that \eqref{defUniformDistr} is equivalent to
\begin{align}\label{defUniformDistr1}
 \lim\limits_{N\rightarrow\infty}\frac{1}{N^{2}}\sum\limits_{k,j=1}^{N}
 e^{2\pi i \langle w,\mathbf{x}_{j}-\mathbf{x}_{k} \rangle}=0 \ \ \ \mathrm{for \ all} \ w\in\Lambda^{*}\setminus\{0\}.
\end{align}

We will use the definition of hyperuniformity in terms of number variance.

\begin{definition}\label{def1}{(Number variance)}
  Let $(X_{N})_{N\in \mathbb{N}}$ be a sequence of point sets  on the flat tori. The  number variance of the sequence for balls of opening radius $R$ is given by
  \begin{align}\label{VarianceDef}
&V(X_{N},R):=
\mathbb{V}_{x}\#(X_{N}\cap B(\mathbf{x},R)) \notag \\
&=
\int\limits_{\Omega_{\Lambda}}
 \left(\sum\limits_{j=1}^{N}\mathbbm{1}_{B(\mathbf{x}, R)}(\mathbf{x}_{j})-N\mathrm{Vol}\big(B(\mathbf{x}, R)\big)\right)^{2}d\mu(\mathbf{x}).
\end{align}
\end{definition}

A classic measure of uniform distribution is given by the $L^{2}$-discrepancy
\begin{equation}\label{discrepL2}
 D_{N}^{2}(X_{N})
 =\left(\int\limits_{0}^{\frac{1}{2}\mathrm{diam}\Omega_{\Lambda}} V(X_{N},R) \, dR \right)^{\frac{1}{2}}.
\end{equation}

\begin{definition}\label{def1}{(Hyperuniformity).}
 Let $(X_{N})_{N\in \mathbb{N}}$ be a sequence of point sets in the  $\Omega_{\Lambda}$. A sequence is called
 \begin{itemize}
 \item {\bf hyperuniform for large balls}, if
 \begin{align}\label{hyperuniformLargeBalls}
 V(X_{N},R)=o(N) \ \ \mathrm{as} \ N\rightarrow\infty
\end{align}
for all $R\in(0,\frac{1}{2}\mathrm{diam}\Omega_{\Lambda})$;

 \item {\bf hyperuniform for small balls}, if
 \begin{align}\label{hyperuniformSmallBalls}
 V(X_{N},R)=o(N\mathrm{Vol}(B(\mathbf{x}, R_{N}))) \ \ \mathrm{as} \ N\rightarrow\infty
\end{align}
and all sequences $(R_{N})_{N\in\mathbb{N}}$ such that
\begin{itemize}
\item[•] (1) $\lim\limits_{N\rightarrow\infty}R_{N}=0$
\item[•] (2) $\lim\limits_{N\rightarrow\infty}N\mathrm{Vol}(B(\mathbf{x}, R_{N}))=\infty$, which is equivalent to ${R_{N}N^{\frac{1}{d}}\rightarrow\infty}$;
\end{itemize}

\item {\bf hyperuniform for balls of threshold order}, if
\begin{align}\label{hyperuniformTresholdOrder}
\limsup\limits_{N\rightarrow\infty} V(X_{N},tN^{-\frac{1}{d}})=\mathcal{O}(t^{d-1}) \ \ \mathrm{as} \ t\rightarrow\infty .
\end{align}
 \end{itemize}
\end{definition}

The hyperuniformity on the unit sphere $\mathbb{S}^{d}$ was studied in \cite{BrauchartGrabnerKusner2018} and \cite{BrauchartGrabnerKusnerZiefle2018}.

The Fourier series for the indicator function
 ${g_{\mathbf{y}}(\mathbf{x}):=\mathbbm{1}_{B(\mathbf{y}, R)}(\mathbf{x})}$ of the Euclidean ball  $B(\mathbf{x}, R)$  on the lattice $\Lambda^{*}$ has the form
\begin{align}\label{FourierSeriesIndicatorFunction}
\mathbbm{1}_{B(\mathbf{x}_{j}, R)}(\mathbf{x})=\widehat{g}_{\mathbf{x}_{j}}(0)+\sum\limits_{w\in\Lambda^{*}\setminus\{0\}}\widehat{g}_{\mathbf{x}_{j}}(v)e^{2\pi i\langle w, \mathbf{x} \rangle},
\end{align}
with Fourier coefficients
\begin{align}\label{FourierCoefficientsIndicatorFunction}
\widehat{g}_{\mathbf{x}_{j}}(w)&=\int\limits_{\Omega_{\Lambda}}\mathbbm{1}_{B(\mathbf{x}_{j}, R)}(\mathbf{x})e^{-2\pi i\langle w, \mathbf{x} \rangle}d\mu(\mathbf{x}) 
= \int\limits_{|\mathbf{x}-\mathbf{x}_{j}|\leq R} e^{-2\pi i\langle w, \mathbf{x} \rangle}d\mu(\mathbf{x}) 
= e^{-2\pi i\langle w, \mathbf{x}_{j} \rangle}a_{w}(R),
\end{align}
where
\begin{align}\label{avR}
a_{w}(R):=
\int\limits_{|\mathbf{x}|\leq R} e^{-2\pi i\langle w, \mathbf{x} \rangle}d\mu(\mathbf{x}).
\end{align}

Then, the variance $V(X_{N},R)$ can be written as
\begin{align}
 &V(X_{N},R)=\int\limits_{\Omega_{\Lambda}}\left(
 \sum\limits_{j=1}^{N}\mathbbm{1}_{B(\mathbf{x}, R)}(\mathbf{x}_{j})-N\mathrm{Vol}(B(\mathbf{x}, R))\right)^{2}d\mu(\mathbf{x}) \notag  \\
 &=\sum\limits_{w\in\Lambda^{*}\setminus\{0\}}a_{w}(R)a_{-w}(R)\sum\limits_{k,j=1}^{N}e^{-2\pi i\langle w,\mathbf{x}_{k}-\mathbf{x}_{j} \rangle}  =
 \sum\limits_{w\in\Lambda^{*}\setminus\{0\}}a_{w}(R)a_{-w}(R)\Big|\sum\limits_{j=1}^{N}e^{-2\pi i\langle w,\mathbf{x}_{j} \rangle}\Big|^{2}.  \label{Var11}
\end{align}

The coefficients $a_{w}(R)$ can be computed by integrating in spherical coordinates. Using  the
spherical coordinate system with a radial coordinate $r$ and angular coordinates $\varphi_{1},...,\varphi_{d-1}$,  where the domain of each $\varphi_{j}$, except $\varphi_{d-1}$, is $[0,\pi)$, and the domain of $\varphi_{d-1}$  is $[0, 2\pi)$, we obtain
\begin{equation}\label{avR_sph}
a_{w}(R)=
\int\limits_{0}^{R} \int\limits_{0}^{\pi}...\int\limits_{0}^{2\pi}
e^{-2\pi i  r|w| \cos\varphi_{1}} r^{d-1} \left(\sin \varphi_{1}\right)^{d-2}...
\sin \varphi_{d-2} d\varphi_{1}...d\varphi_{d-1}dr.
\end{equation}

%&=\int\limits_{0}^{R}r^{d-1}\int\limits_{0}^{\pi}e^{-2\pi i r |v|\cos\varphi}  (\sin \varphi)^{d-2}d\varphi dr
%\left(2\int\limits_{0}^{\frac{\pi}{2}} \left(\sin \varphi_{1}\right)^{d-3} \, d\varphi_{1} \right) ...
%\left(2\int\limits_{0}^{\frac{\pi}{2}} \sin \varphi_{d-3} \, d\varphi_{d-3} \right) 

Each of last $d-2$ integrals is a particular value of the  beta--function, which can be rewritten in terms of gamma functions
\begin{align}\label{cd}
&
 \int\limits_{0}^{\pi}...\int\limits_{0}^{\pi}\int\limits_{0}^{2\pi}
 \left(\sin \varphi_{1}\right)^{d-3}... \,
(\sin \varphi_{d-3}) \, d\varphi_{1}... \, d\varphi_{d-2} \notag \\
 &=
 2\pi B\Big(\frac{d-2}{2},\frac{1}{2}\Big) B\Big(\frac{d-3}{2},\frac{1}{2}\Big)... 
 B\Big(1,\frac{1}{2}\Big) 
 =
 \frac{2\pi^{\frac{d-1}{2}}}{\Gamma(\frac{d-1}{2})}.
\end{align}

Thus, from \eqref{avR_sph} and \eqref{cd} it follows, that 
\begin{equation}\label{avR_calc}
a_{w}(R)
= \frac{2\pi^{\frac{d-1}{2}}}{\Gamma(\frac{d-1}{2})}
\int\limits_{0}^{R}\int\limits_{0}^{\pi}e^{-2\pi i r |w|\cos\varphi}  r^{d-1}(\sin \varphi)^{d-2}d\varphi dr.
\end{equation}

Applying relations \eqref{formulaBessel1} and \eqref{formulaBessel2},
we have that
\begin{align}\label{avRCoef}
&\int\limits_{0}^{R}r^{d-1}\int\limits_{0}^{\pi}e^{-2\pi i r |w|\cos\varphi}(\sin \varphi)^{d-2}d\varphi \, dr \notag \\
 = &\frac{\pi^{\frac{1}{2}}\Gamma(\frac{d-1}{2})}{(\pi|w|)^{\frac{d}{2}-1}}\int\limits_{0}^{R}r^{\frac{d}{2}}J_{\frac{d}{2}-1}(2\pi r|w|)dr = \frac{\Gamma(\frac{d-1}{2})}{\pi^{\frac{d-1}{2}}|w|^{\frac{d}{2}}}  R^{\frac{d}{2}}  J_{\frac{d}{2}}(2\pi R|w|).
\end{align}

So, formulas \eqref{avR_calc} and \eqref{avRCoef} imply
\begin{equation}\label{avRcomp}
a_{w}(R)=
R^{\frac{d}{2}} |w|^{-\frac{d}{2}} J_{\frac{d}{2}}(2\pi|w|R).
\end{equation}

Finally, from \eqref{Var11} the variance $V(X_{N},R)$ can be expressed as
\begin{align}\label{Variance}
 V(X_{N},R)=R^{d}
 \sum\limits_{w\in\Lambda^{*}\setminus\{0\}}|w|^{-d} \Big(J_{\frac{d}{2}}(2\pi|w|R) \Big)^{2}
 \sum\limits_{k,j=1}^{N}e^{-2\pi i\langle w,\mathbf{x}_{j}-\mathbf{x}_{k} \rangle}.
\end{align}

\section{Hyperuniformity for large balls, small balls and balls of threshold order }\label{Hyperuniform} 

\subsection{Hyperuniformity for large balls }
\label{LargeBallsHyperuniform} 

\begin{theorem}\label{theorem_largeBalls}
Let $(X_{N})_{N\in\mathbb{N}}$ be a sequence of point sets, which is hyperuniform for large balls. Then for all $w\in \Lambda^{*}\setminus\{0\}$ 
\begin{equation}\label{theooremLargeB}
\lim\limits_{N\rightarrow\infty}\frac{1}{N}\sum\limits_{k,j=1}^{N}
 e^{2\pi i \langle w,\mathbf{x}_{j}-\mathbf{x}_{k} \rangle}=0.
\end{equation}
As a consequence, sequences which are hyperuniform for large balls are uniformly distributed.
\end{theorem}
\begin{proof}[Proof of Theorem  \ref{theorem_largeBalls}]
Using the definition of the hyperuniformity for large balls and \eqref{Variance}, we have that for all $w\in \Lambda^{*}\setminus\{0\}$ and $R\in(0,\frac{1}{2}\mathrm{diam} \Omega_{\Lambda})$
\begin{equation}
0=\lim\limits_{N\rightarrow\infty}\frac{V(X_{N},R)}{N}
\geq 
\! R^{d} |w|^{-d} \Big(J_{\frac{d}{2}}(2\pi|w|R) \Big)^{2}\lim\limits_{N\rightarrow\infty}\frac{1}{N}\sum\limits_{k,j=1}^{N}e^{-2\pi i\langle w, \mathbf{x}_{k}-\mathbf{x}_{j} \rangle},
\end{equation}
which implies \eqref{theooremLargeB}.
\end{proof}

\begin{remark}
Notice, that for values of $R$ for which one of the values $J_{\frac{d}{2}}(2\pi|w_{0}|R)$ vanishes, nothing can be said about the limit \eqref{theooremLargeB} for $w=w_{0}$.  But there are only countably many such values of $R$. Moreover, we can show that at most only one coefficient $J_{\frac{d}{2}}(2\pi|w|R)$  could vanish for a given value of $R\in(0,\frac{1}{2}\mathrm{diam} \Omega_{\Lambda})$.
\end{remark}

\begin{proof}
We construct  a point set, such that \eqref{theooremLargeB} holds for all 
$w\in\Lambda^{*}\setminus\{0, w_{0}\}$ and for $w=w_{0}$ 
\begin{equation}\label{theooremLargeB_Infinity}
\lim\limits_{N\rightarrow\infty}\frac{1}{N}\sum\limits_{k,j=1}^{N}
 e^{2\pi i \langle w_{0},\mathbf{x}_{j}-\mathbf{x}_{k} \rangle}=\infty.
\end{equation}

Let consider a positive measure $d\mu_{0}(\mathbf{x})=\left(1+\frac{1}{2}\cos(2\pi \langle w_{0},\mathbf{x}\rangle) \right)d\mu(\mathbf{x})$ on $\Omega_{\Lambda}$.
For every $L$ there exists an $N(L)\asymp L^{d}$ such that there is a point set $X_{N}$ for which
\begin{equation}\label{design}
\frac{1}{N}\sum\limits_{j=1}^{N}
p(\mathbf{x}_{j})=\int\limits_{\Omega_\Lambda}p(\mathbf{x})d\mu_{0}(\mathbf{x}),
\end{equation}
for all trigonometric polynomials $p$ of degree $\leq L$.

An example of such point set is an $L$--design with $N\asymp L^{d}$ points.
 The existence of $L$--designs consisting of $N$ nodes, for any $N\geq C L^{d}$ on  $d$--dimensional compact connected oriented Riemannian manifold was proved recently in \cite{GariboldiGigante}.
 Let consider $p(\mathbf{x})= e^{2\pi i \langle w ,\mathbf{x}-\mathbf{y} \rangle}$ for fixed $\mathbf{y}\in \Omega_\Lambda$.
 Then for all $w\in\Lambda^{*}\setminus\{0, w_{0}\}$, such that $|w|\leq L$, we get
\begin{equation}
\frac{1}{N}\sum\limits_{j=1}^{N}
e^{2\pi i \langle w ,\mathbf{x}_{j}-\mathbf{y} \rangle}=
0 \ \ \forall \mathbf{y}\in \Omega_\Lambda,
\end{equation} 
so,  \eqref{theooremLargeB} holds.

For $|w|=|w_{0}|\leq L$ for arbitrary $\mathbf{y}\in \Omega_\Lambda$ we obtain
\begin{equation}
\frac{1}{N}\sum\limits_{j=1}^{N}
e^{2\pi i \langle w_{0} ,\mathbf{x}_{j}-\mathbf{y} \rangle}
=\frac{1}{2}\int\limits_{\Omega_\Lambda}e^{2\pi i \langle w_{0} ,\mathbf{x}_{j}-\mathbf{y} \rangle}\cos(2\pi \langle w_{0},\mathbf{x}\rangle) d\mu(\mathbf{x})=\frac{1}{4},
\end{equation}
which yields 
\begin{equation}\label{designv0}
\frac{1}{N}\sum\limits_{k,j=1}^{N}
e^{2\pi i \langle w_{0} ,\mathbf{x}_{k}-\mathbf{x}_{j} \rangle}
=\frac{1}{4}N.
\end{equation}
Relation \eqref{designv0} implies  \eqref{theooremLargeB_Infinity}.
\end{proof}

\subsection{Hyperuniformity for small balls }
\label{SmallBallsHyperuniform} 
\begin{theorem}\label{theorem_smallBalls}
Let $(X_{N})_{N\in\mathbb{N}}$ be a sequence of point sets, which is hyperuniform for small balls. Then $(X_{N})_{N\in\mathbb{N}}$ is uniformly distributed.
\end{theorem}

\begin{proof}[Proof of Theorem  \ref{theorem_smallBalls}]
From the definition of the hyperuniformity for small balls and \eqref{Variance}, we have
\begin{align}\label{seriesSmallBalls}
\lim\limits_{N\rightarrow\infty}\frac{V(X_{N},R_{N})}{N \mathrm{ Vol}(B(\mathbf{x};R_{N}))}
\!=&\!\lim\limits_{N\rightarrow\infty}
R_{N}^{d}
\sum\limits_{w\in\Lambda^{*}\setminus\{0\}}|w|^{-d} 
\frac{\Big(J_{\frac{d}{2}}(2\pi|w|R_{N}) \Big)^{2}}{\mathrm{ Vol}(B(\mathbf{x};R_{N})) }
\frac{1}{N}\sum\limits_{k,j=1}^{N}e^{-2\pi i\langle w,\mathbf{x}_{k}-\mathbf{x}_{j} \rangle}
\notag \\
=\frac{\Gamma(\frac{d}{2}+1)}{\pi^{\frac{d}{2}}} &\lim\limits_{N\rightarrow\infty} 
\sum\limits_{w\in\Lambda^{*}\setminus\{0\}}|w|^{-d} 
\Big(J_{\frac{d}{2}}(2\pi|w|R_{N}) \Big)^{2}
\frac{1}{N}\sum\limits_{k,j=1}^{N}e^{-2\pi i\langle w, \mathbf{x}_{k}-\mathbf{x}_{j} \rangle}=0.
\end{align}

The order of decreasing of $w$-th Fourier coefficient in \eqref{seriesSmallBalls} equals to 
\begin{align}\label{orderCoef}
|w|^{-d} \Big(J_{\frac{d}{2}}(2\pi|w|R_{N}) \Big)^{2}\asymp
\begin{cases}
 R_{N}^{d}, & \text{if }
 |w| \ll \frac{1}{R_{N}}, \\
|w|^{-d-1}R_{N}^{-1} , & \text{if }  |w| \gg \frac{1}{R_{N}},
  \end{cases}
\end{align}
for $R_{N}\rightarrow0$.

Here and further  we use Vinogradov notations $a_{n}\ll b_{n}$ ($a_{n}\gg b_{n}$)  to mean that there exists positive constant $C$ independent of $n$, such  that $a_{n}\leq C b_{n}$ ($a_{n}\geq C b_{n}$) and
we write $a_{n}\asymp b_{n}$ to mean that  $a_{n}\ll b_{n}$  and $a_{n}\gg b_{n}$.

%Here and further we write $a_{n}\asymp b_{n}$ to mean that there exist positive constants $C_{1}$ and $C_{2}$ independent of $n$, such  that $C_{1}a_{n}\leq b_{n}\leq C_{2}a_{n}$ and we use Vinogradov notations $a_{n}\ll b_{n}$ ($a_{n}\gg b_{n}$)  to mean that there exists positive constant $C$ independent of $n$, such  that $a_{n}\leq C b_{n}$ ($a_{n}\geq C b_{n}$).

From \eqref{seriesSmallBalls} and  \eqref{orderCoef} we have that
\begin{equation}
R_{N}^{d}\frac{1}{N}\sum\limits_{k,j=1}^{N}e^{-2\pi i\langle w, \mathbf{x}_{k}-\mathbf{x}_{j} \rangle}=0.
\end{equation}

 Since $R_{N}N^{\frac{1}{d}}\rightarrow\infty$, then from last relation it follows that
\begin{align}\label{form2}
\limsup\limits_{N\rightarrow\infty}\frac{1}{N^{2}}\sum\limits_{k,j=1}^{N}
 e^{2\pi i \langle w,\mathbf{x}_{j}-\mathbf{x}_{k} \rangle}=0
\end{align}
for all $w\in\Lambda^{*}\setminus\{0\}$. Theorem~\ref{theorem_smallBalls} is proved.

\end{proof}

\subsection{Hyperuniformity for balls of threshold order}
\label{SmallBallsHyperuniform} 
\begin{theorem}\label{theorem_thresholdBalls}
Let $(X_{N})_{N\in\mathbb{N}}$ be a sequence of point sets, which is hyperuniform for  balls of threshold order. Then $(X_{N})_{N\in\mathbb{N}}$ is uniformly distributed.
\end{theorem}

\begin{proof}[Proof of Theorem  \ref{theorem_thresholdBalls}]
Using the definition of the hyperuniformity for balls of threshold order and \eqref{Variance}, we obtain
\begin{align}\label{thresh1}
 V(X_{N},tN^{-\frac{1}{d}})
 \gg t^{d}N^{-1} |w|^{-d} \Big(J_{\frac{d}{2}}(2\pi|w| tN^{-\frac{1}{d}}) \Big)^{2}\Big|\sum\limits_{j=1}^{N}e^{-2\pi i\langle w,\mathbf{x}_{j} \rangle}\Big|^{2}. 
\end{align}

For fixed $t>0$, $w\in\Lambda^{*}\setminus\{0\}$ and $N\rightarrow\infty$, the  asymptotic estimate \eqref{asympSmall} implies 
 \begin{align}\label{thresh2}
  V(X_{N},tN^{-\frac{1}{d}})
 \gg   t^{2d}N^{-2}\sum\limits_{k,j=1}^{N}e^{-2\pi i\langle w,\mathbf{x}_{k}-\mathbf{x}_{j} \rangle}. 
\end{align}

%\begin{multline}\label{thresh1}
 %V(X_{N},tN^{-\frac{1}{d}})
 %\gg \sum\limits_{v\in\Lambda\setminus\{0\}}t^{d}N^{-1} |v|^{-d} \Big(J_{\frac{d}{2}}(2\pi|v| tN^{-\frac{1}{d}}) \Big)^{2}\Big|\sum\limits_{j=1}^{N}e^{-2\pi i\langle v,\mathbf{x}_{j} \rangle}\Big|^{2} 
% \\
 %\gg {\mathop{\sum}\limits_{v\in\Lambda\setminus\{0\}, \atop |v|\ll \frac{1}{tN^{-\frac{1}{d}} }}}t^{2d}N^{-2} \Big|\sum\limits_{j=1}^{N}e^{-2\pi i\langle v,\mathbf{x}_{j} \rangle}\Big|^{2} 
% \\
 %+ {\mathop{\sum}\limits_{v\in\Lambda\setminus\{0\}, \atop |v|\gg \frac{1}{tN^{-\frac{1}{d}}}}}t^{d-1}N^{-1+\frac{1}{d}} |v|^{-d-1} \Big|\sum\limits_{j=1}^{N}e^{-2\pi i\langle v,\mathbf{x}_{j} \rangle}\Big|^{2} 
%\end{multline}

So, the relation 
 \begin{align}\label{thresh3}
 \limsup\limits_{N\rightarrow\infty} V(X_{N},tN^{-\frac{1}{d}})=\mathcal{O}(t^{d-1})
\end{align}
holds only, if
\begin{align*}
\limsup\limits_{N\rightarrow\infty}\frac{1}{N^{2}}\sum\limits_{k,j=1}^{N}
 e^{2\pi i \langle w,\mathbf{x}_{k}-\mathbf{x}_{j} \rangle}=0.
\end{align*}
So, the sequence $(X_{N})_{N\in\mathbb{N}}$ is uniformly distributed, and this completes the proof.

\end{proof}

\section{Hyperuniformity of QMC design sequences}
\label{QMCHyperuniform} 

The notion of QMC design sequences for Sobolev spaces $\mathbb{H}^{s}(\mathbb{S}^{d})$ on the unit sphere $\mathbb{S}^{d}$ was introduced in \cite{Brauchart-Saff-Sloan+2014:qmc_designs}. In the same way we will write down the definition of QMC designs for Sobolev classes $W^{\alpha,2}(\Omega_{\Lambda})$ on flat tori.

The Sobolev space $W^{\alpha,2}(\Omega_{\Lambda})$, $\alpha>\frac{d}{2}$, consists of all functions $f$ such that 
\begin{equation}\label{SobolevSpaces}
  \| f\|_{W^{\alpha,2}}:=\left(\sum\limits_{w\in\Lambda^{*}}(1+4\pi^{2}|w|^{2})^{-\alpha}|\hat{f}(w)|^{2} \right)^{\frac{1}{2}}<\infty,
\end{equation}
where
\begin{equation*}
\hat{f}(w)=\int\limits_{\Omega_{\Lambda}}f(u)e^{-2\pi i \langle u, w\rangle} d\mu(u).
\end{equation*}

The worst-case  error of the cubature rule $Q[X_{N}]$ in a
 space $W^{\alpha,2}(\Omega_{\Lambda})$ of continuous functions on $\Omega_{\Lambda}$  is defined by
 \begin{equation}\label{wce}
 \mathrm{wce}(Q[X_{N}];W^{\alpha,2}):={\mathop{\sup}\limits_{f\in W^{\alpha,2}, \atop \|f\|_{W^{\alpha,2}}\leq1}}
 |Q[X_{N}](f)-\mathrm{I}(f)|,
 \end{equation}
 where
\begin{equation*}
\mathrm{I}(f):=\int_{\Omega_{\Lambda}}f(\mathbf{x})d\mu(\mathbf{x}), \ \ \ \
  Q[X_{N}](f):=
  \frac{1}{N}\sum\limits_{i=1}^{N}f(\mathbf{x}_{i}).
\end{equation*} 
 
 It was showen (see, e.g. \cite{BrandChoirColzGigSeriTravQuadrature}) that there exist sequences of point sets $(X_{N})_{N\in\mathbb{N}}$ and $c>0$, such that
  \begin{equation}\label{upper}
 |Q[X_{N}](f)-\mathrm{I}(f)|\leq c N^{-\frac{\alpha}{d}}\|f\|_{W^{\alpha,2}},
 \end{equation}
 and for every $\alpha>\frac{d}{2}$ there exist $c>0$ such that for every distribution of points $X_{N}$ there exists a function $f\in W^{\alpha,2}(\Omega_{\Lambda})$ with
 \begin{equation}\label{lower}
 |Q[X_{N}](f)-\mathrm{I}(f)|\geq c N^{-\frac{\alpha}{d}}\|f\|_{W^{\alpha,2}}.
 \end{equation}

Analogs of inequalities \eqref{upper} and \eqref{lower} for  spaces $\mathbb{H}^{s}(\mathbb{S}^{d})$ on the unit sphere $\mathbb{S}^{d}$ were obtained in \cite{BrauchartHesse2007},  \cite{Hesse2006}--\cite{HesseSloan2006}.

\begin{definition} Given $\alpha>\frac{d}{2}$, a sequence $(X_{N})_{N\in\mathbb{N}}$ of $N$--point
   configurations in $\Omega_{\Lambda}$ with $N\rightarrow\infty$ is said to be a
   sequence of QMC designs for $W^{\alpha,2}(\Omega_{\Lambda})$ if there exists
   $c(\alpha,d)>0$, independent of $N$, such that
\begin{equation}\label{QmcHsDefinition}
 |Q[X_{N}](f)-\mathrm{I}(f)|\leq c(\alpha,d) N^{-\frac{\alpha}{d}}\|f\|_{W^{\alpha,2}} \ \ \ \mathrm{for \ all} \ f\in W^{\alpha,2}(\Omega_{\Lambda}).
\end{equation}
\end{definition}

Since the point-evaluation
functional is bounded in the space of real-valued functions $W^{\alpha,2}(\Omega_{\Lambda})$  whenever $\alpha>\frac{d}{2}$, the Riesz representation theorem assures
the existence of a reproducing kernel, which can be written in the form
\begin{align}\label{reproducingKernel}
  K_{\Lambda,\alpha}(\mathbf{x},\mathbf{y})=\sum\limits_{w\in\Lambda^{*}}(1+4\pi^{2}|w|^{2})^{-\alpha}e^{2\pi i\langle w, \mathbf{x}-\mathbf{y}\rangle}  =
 \sum\limits_{w\in\Lambda^{*}}(1+4\pi^{2}|w|^{2})^{-\alpha}\cos(2\pi\langle w, \mathbf{x}-\mathbf{y}\rangle).
\end{align}
It can be easy verified that the kernel $K_{\Lambda,\alpha}(\mathbf{x},\mathbf{y})$ has the reproducing kernel properties:
(i)
$K_{\Lambda,\alpha}(\mathbf{x},\mathbf{y})=K_{\Lambda,\alpha}(\mathbf{y},\mathbf{x})$ for
all $\mathbf{x},\mathbf{y}\in \Omega_{\Lambda}$; (ii)
$K_{\Lambda,\alpha}(\cdot,\mathbf{x})\in
 W^{\alpha,2}(\Omega_{\Lambda})$ for all fixed
$\mathbf{x}\in W^{\alpha,2}(\Omega_{\Lambda})$; and (iii) the
reproducing property
\begin{equation*}
(f,K_{\Lambda,\alpha}(\cdot,\mathbf{x}))_{ W^{\alpha,2}}=f(\mathbf{x})
\quad \forall f\in W^{\alpha,2}(\Omega_{\Lambda}) \quad \forall
\mathbf{x} \in \Omega_{\Lambda}.
\end{equation*}

Using arguments, as in \cite{Brauchart-Saff-Sloan+2014:qmc_designs}, it is possible to write down a
 computable expression for the worst-case error. Indeed
 \begin{align}
 \mathrm{wce}(Q[X_{N}];W^{\alpha,2}(\Omega_{\Lambda}))^{2}
&={\mathop{\sup}\limits_{f\in W^{\alpha,2}, \atop \|f\|_{W^{\alpha,2}}\leq1}}
\left(f, \frac{1}{N}\sum\limits_{j=1}^{N}K_{\Lambda,\alpha}(\cdot, \mathbf{x}_{j})- \int\limits_{\Omega_{\Lambda}}K_{\Lambda,\alpha}(\cdot, \mathbf{x})d\mu(\mathbf{x}) \right) \notag
\\
&=\left\|\frac{1}{N}\sum\limits_{j=1}^{N}K_{\Lambda,\alpha}(\cdot, \mathbf{x}_{j})- \int\limits_{\Omega_{\Lambda}}K_{\Lambda,\alpha}(\cdot, \mathbf{x})d\mu(\mathbf{x}) \right\|^{2}_{W^{\alpha,2}} \label{expressionWorstCase} \\
&=\frac{1}{N^{2}}\sum\limits_{w\in\Lambda^{*}\setminus\{0\}}(1+4\pi^{2}|w|^{2})^{-\alpha}\sum\limits_{k,j=1}^{N}e^{2\pi i\langle w, \mathbf{x}_{k}-\mathbf{x}_{j}\rangle}, \notag
 \end{align}
 where we have used the reproducing property of $K_{\Lambda,\alpha}$.

 \begin{theorem}\label{theorem_QMC}
Let $(X_{N})_{N\in\mathbb{N}}$ be a QMC design sequence for $W^{\alpha,2}(\Omega_{\Lambda})$, ${\alpha>\frac{d+1}{2}}$. Then $(X_{N})_{N\in\mathbb{N}}$ is hyperuniform for large balls, small balls and balls of threshold order.
\end{theorem}

Before proving of Theorem~\ref{theorem_QMC} we show that the following lemma takes place.

\begin{lemma}\label{lem1} For any $N$-point set $X_{N}\in \Omega_{\Lambda}$ and $R\in(0,\frac{1}{2}\mathrm{diam}\Omega_{\Lambda})$ 
  the relation
\begin{align}\label{Lemma1}
V(X_{N},R)\ll R^{d-1}N^{2} \mathrm{wce}(Q[X_{N}]; W^{\frac{d+1}{2},2})^{2}
\end{align}
holds.
\end{lemma}
\begin{proof}[Proof of Lemma  \ref{lem1}]
From \eqref{Variance}, \eqref{orderCoef} and \eqref{expressionWorstCase}  we have that
\begin{align*}
 V(X_{N},R)
 &\ll R^{2d}{\mathop{\sum}\limits_{w\in\Lambda^{*}\setminus\{0\}, \atop |w|\leq \frac{1}{R} }} \sum\limits_{k,j=1}^{N}e^{-2\pi i\langle w,\mathbf{x}_{k} -\mathbf{x}_{j}\rangle}
 +R^{d-1} {\mathop{\sum}\limits_{w\in\Lambda^{*}\setminus\{0\}, \atop |w|> \frac{1}{R}}} |w|^{-d-1}  \sum\limits_{k,j=1}^{N}e^{-2\pi i\langle w,\mathbf{x}_{k}-\mathbf{x}_{j} \rangle}  \notag \\
& \ll
 R^{d-1}\sum\limits_{w\in\Lambda^{*}\setminus\{0\}} |w|^{-d-1} \sum\limits_{j=1}^{N}e^{-2\pi i\langle w,\mathbf{x}_{k}-\mathbf{x}_{j} \rangle} \ll R^{d-1}N^{2} \mathrm{wce}(Q[X_{N}]; W^{\frac{d+1}{2},2})^{2}.
\end{align*}
Lemma \ref{lem1} is proved.
\end{proof}

Notice that in the same way as it was shown for Sobolev spaces $\mathbb{H}^{s}(\mathbb{S}^{d})$
on the unit sphere $\mathbb{S}^{d}$ (see Theorem 9 in \cite{Brauchart-Saff-Sloan+2014:qmc_designs} we can prove that the following statement is true
\begin{lemma}\label{lem2} Given $\alpha>\frac{d}{2}$, let $(X_{N})_N$ be a sequence of
  QMC designs for
 $W^{\alpha,2}(\Omega_{\Lambda})$. Then $(X_{N})_{N}$ is a sequence of QMC designs
  for $W^{\beta,2}(\Omega_{\Lambda})$, for all
  $\frac{d}{2}<\beta\leq\alpha$.
  \end{lemma}

\begin{proof}[Proof of Theorem  \ref{theorem_QMC}]
Let $(X_{N})_{N\in\mathbb{N}}$ be a QMC design sequence for $W^{\alpha,2}(\Omega_{\Lambda})$, $\alpha\geq \frac{d+1}{2}$, then by Lemma~\ref{lem2} it is a QMC sequence for  $W^{\frac{d+1}{2},2}(\Omega_{\Lambda})$. So
\begin{equation}\label{CorLemma1}
\mathrm{wce}(Q[X_{N}]; W^{\frac{d+1}{2},2})^{2}\ll N^{-\frac{d+1}{d}}.
\end{equation}
Then, Lemma~\ref{lem1}  and  \eqref{CorLemma1} imply that for any $R\in(0,\frac{1}{2}\mathrm{diam}\Omega_{\Lambda})$ 
\begin{equation}\label{form31}
 V(X_{N},R)
 \ll R^{d-1}N^{1-\frac{1}{d}}.
\end{equation}

{\bf (i) Large ball regime:}
 It follows immediately from \eqref{form31}, that  for any $R\in(0,\frac{1}{2}\mathrm{diam}\Omega_{\Lambda})$
\begin{equation}\label{form32}
 V(X_{N},R)
 =o(N) \ \ \ \mathrm{as} \ \ N\rightarrow\infty,
\end{equation}
so $(X_{N})_{N\in A}$ is hyperuniform for large balls.

{\bf (ii) Small ball regime:} Let $\lim\limits_{N\rightarrow\infty}R_{N}=0$ and
\begin{equation*}
\lim\limits_{N\rightarrow\infty}NVol(B(\cdot;R_{N}))=\frac{\pi^{\frac{d}{2}}}{\Gamma(\frac{d}{2}+1)}\lim\limits_{N\rightarrow\infty}NR_{N}^{d}=\infty.
\end{equation*}
Then, formula   \eqref{form31} yields
\begin{equation*}
V(X_{N},R_{N})= \mathcal{O}(N^{1-\frac{1}{d}}R_{N}^{-1} \mathrm{Vol}(B(\cdot, R_{N})))=o(NR_{N}^{d})  \ \mathrm{as} \ \ N\rightarrow\infty.
\end{equation*}
This proves the hyperuniformity for small balls.
%so $(X_{N})_{N\in A}$ is hyperuniform for small balls.

{\bf (ii) Threshold regime:} Let $R_{N}=tN^{-\frac{1}{d}}$, $R\in(0,\frac{1}{2}\mathrm{diam}\Omega_{\Lambda})$ and $t>0$.
From \eqref{form31} we have
\begin{equation}\label{form33}
 V(X_{N},R)
 \ll (tN^{-\frac{1}{d}})^{d-1}N^{1-\frac{1}{d}}=t^{d-1} \ \ \mathrm{as}   \ \ N\rightarrow\infty.
\end{equation}

Since $t>0$ was arbitrary, then
\begin{equation}
 \limsup\limits_{N\rightarrow\infty}V(X_{N},tN^{-\frac{1}{d}})
=\mathcal{O}(t^{d-1})  \ \ \mathrm{as} \ \ t\rightarrow\infty,
\end{equation}
so $(X_{N})_{N\in A}$ is hyperuniform also for balls of threshold order.

Theorem~\ref{theorem_QMC} is proved.
\end{proof}

\section{Hyperuniformity of jittered sampling point process}
\label{JitteredSamplings} 

Let consider an area--regular partitions $\mathcal{A}=\{A_{1}, ..., A_{N}\}$ with $\cup_{j=1}^{N}A_{j}=\Omega_{\Lambda}$ and $A_{j}\cap A_{k}=\emptyset$, $k\neq j$ with small 
diameters:
 $\mathrm{diam}(A_{i})\leq C N^{-\frac{1}{d}}$ for $i=1,...,N$. Here $C$ is a constant that does not depend on $N$. The existance of such partitions was shown in \cite{GiganteLeopardi2017:Diameter}.

The jittered sampling variance integral can be written as:
\begin{multline}\label{jitteredVariance}
V(X_{N},R) \\
=\int\limits_{\Omega_{\Lambda}}\int\limits_{A_{1}}...\int\limits_{A_{N}}
 \left(\sum\limits_{j=1}^{N}\mathbbm{1}_{B(\cdot, R)}(\mathbf{x}_{j})-N\mathrm{Vol}(B(\mathbf{x}, R))\right)^{2}d\mu_{1}(\mathbf{x}_{1})...
 d\mu_{N}(\mathbf{x}_{N})d\mu(\mathbf{x}),
\end{multline}
where $\mu_{j}(\cdot)=N\mu(\cdot\cap A_{j})$ is a probability measure.

It is not hard to see that
\begin{align}\label{differ}
&\int\limits_{\Omega_{\Lambda}}
 \left(\sum\limits_{j=1}^{N}\mathbbm{1}_{B(\mathbf{x}, R)}(\mathbf{x}_{j})-N\mathrm{Vol}(B(\cdot, R))\right)^{2} d\mu(\mathbf{x})  \notag \\
& = \int\limits_{\Omega_{\Lambda}}
 \sum\limits_{k,j=1}^{N}\mathbbm{1}_{B(\mathbf{x}_{j}, R)}(\mathbf{x})\mathbbm{1}_{B(\mathbf{x}_{k}, R)}(\mathbf{x})d\mu(\mathbf{x}) 
 -N^{2}\mathrm{Vol}(B(\cdot, R))^{2} \notag \\
 &= \sum\limits_{k,j=1}^{N}\mathrm{Vol}(B(\mathbf{x}_{j},R)\cap B(\mathbf{x}_{k},R))
 -N^{2}\mathrm{Vol}(B(\cdot, R))^{2}.
\end{align}

Taking into account \eqref{differ} and integrating with respect to the probability measure $d\mu_{1}(\mathbf{x}_{1})...
 d\mu_{N}(\mathbf{x}_{N})$, we obtain
 \begin{align}
& V(X_{N},R) \notag \\
&=\int\limits_{A_{1}}...\int\limits_{A_{N}}
 \sum\limits_{k,j=1}^{N}\mathrm{Vol}(B(\mathbf{x}_{j},R)\cap B(\mathbf{x}_{k},R))
d\mu_{1}(\mathbf{x}_{1})...
 d\mu_{N}(\mathbf{x}_{N})
  -N^{2}\mathrm{Vol}(B(\cdot, R))^{2} \notag\\
 & =\mathop{\sum}\limits_{k,j=1, \atop k\neq j}^{N}\int\limits_{A_{k}}\int\limits_{A_{j}}
\mathrm{Vol}(B(\mathbf{x},R)\cap B(\mathbf{y},R))
d\mu_{k}(\mathbf{x})
 d\mu_{j}(\mathbf{y})
  +N\mathrm{Vol}(B(\cdot, R))\notag  \\
&  -N^{2}\mathrm{Vol}(B(\cdot, R))^{2}=
  N^{2}\int\limits_{\Omega_{\Lambda}}\int\limits_{\Omega_{\Lambda}}
  \mathrm{Vol}(B(\mathbf{x},R)\cap B(\mathbf{y},R))d\mu(\mathbf{x}) d\mu(\mathbf{y}) \label{jitteredVariance1} \\
& -
 \sum\limits_{k=1}^{N}\int\limits_{A_{k}}\int\limits_{A_{k}}
\mathrm{Vol}(B(\mathbf{x},R)\cap B(\mathbf{y},R))
d\mu_{k}(\mathbf{x})
 d\mu_{k}(\mathbf{y})
 \! +\!\N\mathrm{Vol}(B(\cdot, R))
 \!- \!N^{2}\mathrm{Vol}(B(\cdot, R))^{2}. \notag
\end{align}
Noticing that
\begin{align*}
&\int\limits_{\Omega_{\Lambda}}\int\limits_{\Omega_{\Lambda}}
  \mathrm{Vol}(B(\mathbf{x},R)\cap B(\mathbf{y},R))d\mu(\mathbf{x}) d\mu(\mathbf{y}) \notag \\
  =&\int\limits_{\Omega_{\Lambda}}\int\limits_{\Omega_{\Lambda}} \int\limits_{\Omega_{\Lambda}}
\mathbbm{1}_{B(\mathbf{x},R)}(\mathbf{z}) \mathbbm{1}_{B(\mathbf{y},R)}(\mathbf{z}) d\mu(\mathbf{z}) d\mu(\mathbf{x})  d\mu(\mathbf{y}) \notag \\
  =&\int\limits_{\Omega_{\Lambda}}\int\limits_{\Omega_{\Lambda}} \int\limits_{\Omega_{\Lambda}}
\mathbbm{1}_{B(\mathbf{z},R)}(\mathbf{x}) \mathbbm{1}_{B(\mathbf{z},R)}(\mathbf{y}) d\mu(\mathbf{x}) d\mu(\mathbf{y})  d\mu(\mathbf{z}) = \mathrm{Vol}(B(\cdot, R))^{2},
\end{align*}
from \eqref{jitteredVariance1} we have
\begin{align}\label{jitteredVariance2}
V(X_{N},R)=
N\mathrm{Vol}(B(\cdot, R))-
 \sum\limits_{k=1}^{N}\int\limits_{A_{k}}\int\limits_{A_{k}}
\mathrm{Vol}(B(\mathbf{x},R)\cap B(\mathbf{y},R))
d\mu_{k}(\mathbf{x})d\mu_{k}(\mathbf{y}).
\end{align}
Using that fact that for any $\mathbf{x}, \mathbf{y}\in \mathbb{R}^{d}:$
 $|\mathbf{x}-\mathbf{y}|\leq 2R$ the following relation holds
 \begin{equation}\label{ineqIntersectionVolumes}
\mathrm{Vol}(B(\mathbf{x},R)\cap B(\mathbf{y},R))\geq 
\mathrm{Vol}\Big(B\Big(\cdot,R-\frac{|\mathbf{x}-\mathbf{y}|}{2}\Big)\Big)
\end{equation}
and also  that for $\mathbf{x}, \mathbf{y}\in A_{k}:$
 $|\mathbf{x}-\mathbf{y}|\leq C_{d}N^{-\frac{1}{d}}$, we obtain the estimate
 \begin{align}\label{jitteredVariance3}
V(X_{N},R)&\leq
N\mathrm{Vol}(B(\cdot, R))-
N\mathrm{Vol}\Big(B\Big(\cdot,R-\frac{|\mathbf{x}-\mathbf{y}|}{2}\Big)\Big) \notag \\
&=\frac{\pi^{\frac{d}{2}}}{\Gamma(\frac{d}{2}+1)}N\Big(R^{d}-\Big(R-\frac{|\mathbf{x}-\mathbf{y}|}{2}\Big)^{d}\Big) \ll N^{1-\frac{1}{d}}R^{d-1}.
\end{align}

Then, from \eqref{jitteredVariance3} and proof of Theorem~\ref{theorem_QMC} we get the following statement.
\begin{theorem}\label{theorem_jittered}
The jittered sampling point process is hyperuniform in all three regimes.
\end{theorem}
%\begin{proof}[Proof of Theorem  \ref{theorem_jittered}]
%From (\ref{jitteredVariance3}) we have that for large balls
%\begin{align*}
%V(X_{N},R)=\mathcal{O}(N^{1-\frac{1}{d}})=o(N) \ \ \forall R\in (0, \frac{1}{2}\mathrm{diam}\Omega_{\Lambda}),
%\end{align*}
%which implies the hyperuniformity for large balls.
 
% From (\ref{jitteredVariance3}) for small balls we have
%\begin{multline*}
%V(X_{N},R_{N})=\mathcal{O}(N^{1-\frac{1}{d}}R_{N}^{d-1})= \mathcal{O}(N^{1-\frac{1}{d}}R_{N}^{-1} \mathrm{Vol}(B(\mathbf{x}, R_{N}))) \\
%= \mathcal{O}(N\mathrm{Vol}(B(\mathbf{x}, R_{N}))),
%\end{multline*}
%where 
 %$\lim\limits_{N\rightarrow\infty}R_{N}=0$ and
 %$\lim\limits_{N\rightarrow\infty}N\mathrm{Vol}(B(\mathbf{x}, R_{N}))=\infty$.
%This proves the hyperuniformity for small balls.

%Substitution $R_{N}=tN^{-\frac{1}{d}}$ into (\ref{jitteredVariance3})  yields
%\begin{align*}
%V(X_{N},tN^{-\frac{1}{d}})=\mathcal{O}(N^{1-\frac{1}{d}}(tN^{-\frac{1}{d}})^{d-1})
%=\mathcal{O}(t^{d-1}).
%\end{align*}
%From the definition of the hyperuniformity  we have that 
%jittered sampling point process is hyperuniform for balls of threshold order.
%\end{proof}

\section{Hyperuniformity of some determinantal processes}
\label{determinantalProcess} 

\begin{definition}
A random point process (see, e.g., Chap. 4 in \cite{ZerosOfGaussianAnalyticProvidence2009}) is called determinantal with kernel
 $K: \Omega_{\Lambda}\times \Omega_{\Lambda}\rightarrow \mathbb{C} $ if it is simple and the joint intensities with respect to a background measure $\mu$ are given by
\begin{align*}
\rho(\mathbf{x}_{1},..., \mathbf{x}_{k})=\mathrm{det}(K(\mathbf{x}_{i},\mathbf{x}_{j}))_{1\leq i,j\leq k},
\end{align*} 
for every $k\geq 1$ and $\mathbf{x}_{1},..., \mathbf{x}_{k}\in  \Omega_{\Lambda}$. 
 \end{definition}

In \cite{ZerosOfGaussianAnalyticProvidence2009} it is shown that a determinantal process samples exactly $N$ points if and only if it is associated to the projection of $L^{2}$ to an $N$--dimensional subspace $H$. Let $\psi_{1},...,\psi_{N}$ be an orthornormal basis of $H$, then the kernel is given by
 \begin{equation}\label{kernelNpoints}
K_{H}(\mathbf{x},\mathbf{y}) = \sum\limits_{k=1}^{N}
\psi_{k}(\mathbf{x})\overline{\psi_{k}(\mathbf{y})}.
\end{equation} 
 
\begin{definition}
We say that $K$ is a projection kernel if it is a Hermitian projection kernel; i.e., the integral operator in $L^{2}(\mu)$ with kernel $K$ is self--adjoint and has eigenvalues $1$ and $0$.
\end{definition} 
 Then by Macchi--Soshnikov's theorem (see, e.g., Theorem 4.5.5 in \cite{ZerosOfGaussianAnalyticProvidence2009})  the projection kernel $K$ defines a determinantal process.
 
We will study the hyperuniformity of point sets, which are drawn from  the determinantal point processes given by similar kernels as in \cite{MarzoOrtegsCerda}
%We define the kernels
 \begin{equation}\label{kernel}
K_{N}(\mathbf{x},\mathbf{y}) =K_{N}(\mathbf{x}-\mathbf{y}) := \sum\limits_{w\in \Lambda^{*}}\kappa_{N}(w)e^{2\pi i\langle \mathbf{x}-\mathbf{y},w \rangle}, \ \ \  \mathbf{x},\mathbf{y}\in \Omega_{\Lambda},
\end{equation} 
where the functions $\kappa=(\kappa_{N})_{N\geq 0}$ have a finite support and are such, that \linebreak ${\kappa_{N}: \Lambda^{*}\rightarrow\{0,1\}}$.
 
Then, for these kernels we have that
  \begin{equation}\label{trace}
tr(K_{N})=\int\limits_{\Omega_{\Lambda}}
K_{N}(\mathbf{x},\mathbf{x}) d\mu(\mathbf{x})
= \sum\limits_{w\in \Lambda^{*}}\kappa_{N}(w) =\# \mathrm{supp}\, \kappa_{N}.
\end{equation} 
 
 So, from \eqref{kernelNpoints} it follows, that the determinantal process, which is defined by the kernel \eqref{kernel}, has $N$ points, if
   \begin{equation*}
\int\limits_{\Omega_{\Lambda}}K_{N}(\mathbf{x},\mathbf{x}) \,d\mu(\mathbf{x})
=N.
\end{equation*} 
 
From \eqref{VarianceDef} and \eqref{differ} we obtain
\begin{equation}\label{ff1}
V(X_{N},R)
 = \sum\limits_{k,j=1}^{N}\mathrm{Vol}(B(\mathbf{x}_{j},R)\cap B(\mathbf{x}_{k},R))
 -N^{2}\mathrm{Vol}(B(\cdot, R))^{2}
=\sum\limits_{k,j=1}^{N} g_{R}(\mathbf{x}_{k},\mathbf{x}_{j}), 
\end{equation}
%Let consider the difference
where
\begin{equation}\label{function_g}
g_{R}(\mathbf{x},\mathbf{y})=g_{R}(\mathbf{x}-\mathbf{y}):=\mathrm{Vol}(B(\mathbf{x},R)\cap B(\mathbf{y},R))-
\mathrm{Vol}(B(\cdot,R))^{2}.
\end{equation}

Formulas \eqref{FourierSeriesIndicatorFunction}, \eqref{FourierCoefficientsIndicatorFunction} and \eqref{avRcomp} allow to write
\begin{align}\label{function_gSeries}
& g_{R}(\mathbf{x},\mathbf{y})=
 \int\limits_{\Omega_{\Lambda}}
\mathbbm{1}_{B(\mathbf{x}, R)}(\mathbf{z})\mathbbm{1}_{B(\mathbf{y}, R)}(\mathbf{z})d\mu(\mathbf{z}) -\mathrm{Vol}(B(\cdot,R))^{2}
 \notag \\
&=
R^{d} \sum\limits_{w\in\Lambda^{*}\setminus\{0\}}|w|^{-d} \Big(J_{\frac{d}{2}}(2\pi|w|R) \Big)^{2} e^{-2\pi i\langle w,\mathbf{x}- \mathbf{y} \rangle}.
\end{align}

Then, the variance for the determinantal point process equals to
\begin{equation}\label{vardeterm1}
V(X_{N},R)=\mathbb{E}\sum\limits_{k,j=1}^{N}g_{R}(\mathbf{x}_{k},\mathbf{x}_{j})= Ng_{R}(0)+\mathbb{E}\mathop{\sum}\limits_{k,j=1, \atop k\neq j}^{N}g_{R}(\mathbf{x}_{k},\mathbf{x}_{j}).
\end{equation} 

To compute the expected value in the last formula we will use the following statement (see e.g., formula 1.2.2 from \cite{ZerosOfGaussianAnalyticProvidence2009}).
\begin{proposition}\label{propositionDeterminantalProcess}
Let $K(\mathbf{x},\mathbf{y})$ be a projection kernel with trace $N$ in $\Omega$, and let ${X_{N}=(\mathbf{x}_{1},..., \mathbf{x}_{N})\in \Omega^{N}}$ be $N$ random points generated by the corresponding determinantal point process. Then, for any measurable $f: \Omega\times\Omega\rightarrow[0,\infty)$, we have
\begin{equation}\label{prop1}
\mathbb{E}_{X_{N}\in\Omega^{N}}\left(\sum\limits_{k\neq j}f(\mathbf{x}_{k},\mathbf{x}_{j})\right)
=\int\limits_{\Omega}\int\limits_{\Omega}\left(K(\mathbf{x},\mathbf{x}) K(\mathbf{y},\mathbf{y})-|K(\mathbf{x},\mathbf{y})|^{2} \right)f(\mathbf{x},\mathbf{y})d\mu(\mathbf{x})d\mu(\mathbf{y}). 
\end{equation}
\end{proposition}

Applying Proposition~\ref{propositionDeterminantalProcess} for \eqref{vardeterm1} and using translation invariance, we get
\begin{align}
&V(X_{N},R)=\int\limits_{\Omega_{\Lambda}}\int\limits_{\Omega_{\Lambda}}\left(K_{N}(\mathbf{x},\mathbf{x}) K_{N}(\mathbf{y},\mathbf{y})-|K_{N}(\mathbf{x},\mathbf{y})|^{2} \right)g_{R}(\mathbf{x},\mathbf{y})d\mu(\mathbf{x})d\mu(\mathbf{y}) \notag \\
&+ Ng_{R}(0)
=Ng_{R}(0)-\int\limits_{\Omega_{\Lambda}}\int\limits_{\Omega_{\Lambda}}|K_{N}(\mathbf{x}-\mathbf{y})|^{2} g_{R}(\mathbf{x}-\mathbf{y})d\mu(\mathbf{x})d\mu(\mathbf{y}) \label{vardeterm2} \\
&=Ng_{R}(0)-\int\limits_{\Omega_{\Lambda}}|K_{N}(\mathbf{x})|^{2} g_{R}(\mathbf{x})d\mu(\mathbf{x}). \notag
\end{align}

Observing that 
\begin{equation}\label{kernelSquared}
|K_{N}(\mathbf{x})|^{2} =\sum\limits_{w,w'\in \Lambda^{*}}\kappa_{N}(w)\kappa_{N}(w')e^{2\pi i\langle \mathbf{x},w-w' \rangle} ,
\end{equation}
we get
\begin{align}
& \int\limits_{\Omega_{\Lambda}}|K_{N}(\mathbf{x})|^{2} g_{R}(\mathbf{x})d\mu(\mathbf{x}) \notag \\
&=
\!\!\sum\limits_{w,w'\in \Lambda^{*}}\!\!\!\kappa_{N}(w)\kappa_{N}(w') \!\!\!\sum\limits_{\xi\in\Lambda^{*}\setminus\{0\}}\!\!\!R^{d}|\xi|^{-d} \Big(J_{\frac{d}{2}}(2\pi|\xi|R) \Big)^{2} 
\!\!\int\limits_{\Omega_{\Lambda}} \!\!e^{2\pi i\langle \mathbf{x}, w-w' \rangle}  
 e^{-2\pi i\langle \mathbf{x}, \xi \rangle}d\mu(\mathbf{x})  \label{inteVardeterm} \\
&= R^{d}\mathop{\sum}\limits_{w,w'\in \Lambda^{*},  \atop w\neq w'}\kappa_{N}(w)\kappa_{N}(w') |w-w'|^{-d} \Big(J_{\frac{d}{2}}(2\pi|w-w'|R) \Big)^{2}, \notag
\end{align} 
where we have used that
\begin{equation}
\int\limits_{\Omega_{\Lambda}} e^{2\pi i\langle \mathbf{x},w-w' \rangle}  
 e^{-2\pi i\langle \mathbf{x}, \xi \rangle}d\mu(\mathbf{x})=\delta_{w'-w,\,\xi}, \ \ w,w'\in\Lambda^{*}, \ \ \ \xi\in \Lambda^{*}\setminus\{0\}. 
\end{equation}

Let $\mathcal{D}_{N}\subset \mathbb{R}^{d}$ be an open subset with  boundary of measure zero, such that $B(\mathbf{0}; c_{0}N^{\frac{1}{d}})\subseteq\mathcal{D}_{N}
\subseteq B(\mathbf{0}; c_{0}N^{\frac{1}{d}}+c_{1})$, where $c_{0}$ and $c_{1}$ are some positive constants.
Define for $N\in\mathbb{N}$ the functions $\kappa_{N}$ in the following way
\begin{equation}\label{kappaN}
\kappa_{N}(w)=
\begin{cases}
 1, & \text{if }
  w\in \mathcal{D}_{N}\cap\Lambda^{*}, \\
0 , & \text{otherwise}.
  \end{cases}
\end{equation}

Observe, that
\begin{equation}\label{trace}
\sum\limits_{w\in \Lambda^{*}}\kappa_{N}(w)=\#\left\{ \mathbb{Z}^{d} \cap
A^{t}\mathcal{D}_{N} \right\}\asymp N.
\end{equation} 

We choose constant $c_{0}$ and $c_{1}$ and the domain $\mathcal{D}_{N}$ in such way that
\begin{align}\label{trace}
\sum\limits_{w\in \Lambda^{*}}\kappa_{N}(w)= N.
\end{align} 

So, from \eqref{function_gSeries},  \eqref{vardeterm2}, \eqref{inteVardeterm} and \eqref{trace} we get
\begin{align}
V(X_{N},R)&=
\sum\limits_{w'\in \Lambda^{*}}\kappa_{N}(w')g_{R}(0) - R^{d}\mathop{\sum}\limits_{w,w'\in \Lambda^{*},  \atop w\neq w'}\kappa_{N}(w)\kappa_{N}(w')|w-w'|^{-d} \Big(J_{\frac{d}{2}}(2\pi|w-w'|R) \Big)^{2} \notag \\
&=
R^{d}\sum\limits_{w'\in \Lambda^{*}}\kappa_{N}(w')\left(
\sum\limits_{w\in \Lambda^{*}\setminus\{0\}}|w|^{-d} \Big(J_{\frac{d}{2}}(2\pi|w|R) \Big)^{2} 
\right. \notag\\
&\left.
-
\mathop{\sum}\limits_{w\in \Lambda^{*},  \atop w\neq w'}|w-w'|^{-d} \Big(J_{\frac{d}{2}}(2\pi|w-w'|R) \Big)^{2}\right) \label{vardeterm3} \\
&=
R^{d}\sum\limits_{w'\in \Lambda^{*}\cap\mathcal{D}_{N}}\left(
\sum\limits_{w\in \Lambda^{*}\setminus\{0\}}|w|^{-d} \Big(J_{\frac{d}{2}}(2\pi|w|R) \Big)^{2} 
\right. \notag\\
&\left.
-
\mathop{\sum}\limits_{w\in \Lambda^{*}\cap\mathcal{D}_{N},  \atop w\neq w'}|w-w'|^{-d} \Big(J_{\frac{d}{2}}(2\pi|w-w'|R) \Big)^{2}\right). \notag
\end{align} 

Taking into account, that for any $w'\in \Lambda^{*}$ and real-valued function $f$
\begin{equation}\label{equalLat}
\sum\limits_{w\in \Lambda^{*}\setminus\{0\}}f(|w|)=
\mathop{\sum}\limits_{w\in \Lambda^{*},  \atop w\neq w'}f(|w-w'|),
\end{equation} 
we have
\begin{equation}\label{vardeterm3}
V(X_{N},R)=
 R^{d}\sum\limits_{w'\in \Lambda^{*}\cap\mathcal{D}_{N}}
\mathop{\sum}\limits_{w\in \Lambda^{*},  w\notin\mathcal{D}_{N},  \atop w\neq w'}
|w-w'|^{-d} \Big(J_{\frac{d}{2}}(2\pi|w-w'|R) \Big)^{2}.
\end{equation}

To estimate the sum from  the last formula we will need the following lemma.
\begin{lemma}\label{lemma_sum} 
Let $M\in \mathbb{R}_{+}$, $a>1$ and $d\geq2$ be fixed. Then
the following estimates hold
\begin{equation}\label{lemma_sumPoints1}
\mathop{\sum}\limits_{w'\in \Lambda^{*}, \atop   |w'|\leq \frac{M}{2}}
\mathop{\sum}\limits_{w\in \Lambda^{*},  \atop |w|\geq M}
|w-w'|^{-d-1}\ll   M^{d-1};
\end{equation} 
\begin{equation}\label{lemma_sumPoints2}
\mathop{\sum}\limits_{w'\in \Lambda^{*}, \atop   \frac{M}{2}<|w'|< M}
\mathop{\sum}\limits_{w\in \Lambda^{*},  \atop M+a\leq|w|<2M}
|w-w'|^{-d-1}\ll   M^{d-1}\ln\left(\frac{a+\frac{M}{2}}{a}\right) ;
\end{equation} 
\begin{equation}\label{lemma_sumPoints3}
\mathop{\sum}\limits_{w'\in \Lambda^{*}, \atop   |w'|< M}
\mathop{\sum}\limits_{w\in \Lambda^{*},  \atop |w|\geq 2M}
|w-w'|^{-d-1}\ll  M^{d-1} .
\end{equation} 
\end{lemma}
\begin{proof}[Proof of Lemma \ref{lemma_sum}]
Before we proceed, we state a simple fact, which we will use repeatedly:
\begin{equation}\label{relat2}
\mathop{\sum}\limits_{u\in \Lambda^{*}, \atop   |u|>t}
|u|^{-d-1}\ll \sum\limits_{k=[t]}^{\infty}\sum\limits_{k\leq|u|<k+1}|u|^{-d-1}\ll  \sum\limits_{k=[t]}^{\infty}k^{-2}\ll t^{-1}.
\end{equation}

Applying \eqref{relat2}, we get
\begin{equation*}
\mathop{\sum}\limits_{w'\in \Lambda^{*}, \atop   |w'|\leq \frac{M}{2}}
\mathop{\sum}\limits_{w\in \Lambda^{*},  \atop |w|\geq M}
|w-w'|^{-d-1}\ll M^{d} \mathop{\sum}\limits_{w\in \Lambda^{*},  \atop |w|\geq \frac{M}{2}}
|w|^{-d-1}
\ll   M^{d-1},
\end{equation*} 
\begin{equation*}
\mathop{\sum}\limits_{w'\in \Lambda^{*}, \atop   |w'|< M}
\mathop{\sum}\limits_{w\in \Lambda^{*},  \atop |w|\geq 2M}
|w-w'|^{-d-1}\ll
M^{d} \mathop{\sum}\limits_{w\in \Lambda^{*},  \atop |w|\geq M}
|w|^{-d-1}
\ll  M^{d-1} .
\end{equation*} 
Thus, the estimates \eqref{lemma_sumPoints1} and \eqref{lemma_sumPoints3} are proved. 

Now, let us show, that \eqref{lemma_sumPoints2} is true.

Fix $w'$ such that $\frac{M}{2}<|w'|<M$. Assume that  $M-k\leq|w'|<M-k+1$, where $k\in\{1, 2,..., [\frac{M}{2}] \}$. Then,
\begin{equation*}
\{w: \ M+a\leq |w|\leq 2M\}\subset \{w: \ |w-w'|\geq a+k-1 \}.
\end{equation*} 

Therefore, on basis of \eqref{relat2} for any such $w'$:
\begin{equation}\label{relat3}
\mathop{\sum}\limits_{w\in \Lambda^{*},  \atop M+a\leq|w|<2M}
|w-w'|^{-d-1}\ll \mathop{\sum}\limits_{w\in \Lambda^{*},  \atop |w-w'|\geq a+k-1}
|w-w'|^{-d-1}=\mathop{\sum}\limits_{u\in \Lambda^{*},  \atop |u|\geq a+k-1}
|u|^{-d-1}\ll (a+k-1)^{-1}.
\end{equation} 
Hence, \eqref{relat3} yields 
\begin{align*}
&\mathop{\sum}\limits_{w'\in \Lambda^{*}, \atop   \frac{M}{2}<|w'|< M}
\mathop{\sum}\limits_{w\in \Lambda^{*},  \atop M+a\leq|w|<2M}
|w-w'|^{-d-1}\ll    \sum\limits_{k=1}^{[\frac{M}{2}]}  
\mathop{\sum}\limits_{w'\in \Lambda^{*},  \atop M-k\leq|w'|<M-k+1} \mathop{\sum}\limits_{w\in \Lambda^{*},  \atop M+a\leq|w|<2M}
|w-w'|^{-d-1} \notag \\
& \ll \sum\limits_{k=1}^{[\frac{M}{2}]}  (M-k)^{d-1} (a+k-1)^{-1} \ll M^{d-1}\ln\left(\frac{a+\frac{M}{2}}{a}\right).
\end{align*}

Lemma~\ref{lemma_sum} is proved.
\end{proof}

From \eqref{orderCoef} and \eqref{lemma_sumPoints3} it follows that 
\begin{align}\label{ReminderVar}
&\sum\limits_{w'\in \Lambda^{*}\cap\mathcal{D}_{N}}
\mathop{\sum}\limits_{w\in \Lambda^{*},   \atop |w|\geq 2c_{0} N^{\frac{1}{d}}}
|w-w'|^{-d} \Big(J_{\frac{d}{2}}(2\pi|w-w'|R) \Big)^{2}
\notag \\
\ll R^{-1}&\sum\limits_{w'\in \Lambda^{*}\cap\mathcal{D}_{N}}
\mathop{\sum}\limits_{w\in \Lambda^{*},    \atop |w|\geq 2c_{0}  N^{\frac{1}{d}}}
|w-w'|^{-d-1} \ll
 R^{-1 } N^{1-\frac{1}{d}}.
\end{align} 
%where we have used, that from $|w|\geq 2c_{0}N^{\frac{1}{d}}, w\in \Lambda^{*}$, it  follows that 
%${|w-w'|\geq c_{0}N^{\frac{1}{d}} -1}$ for $w'\in \Lambda^{*}\cap\mathcal{D}_{N}$.

Combining \eqref{vardeterm3} and \eqref{ReminderVar}, we get 
\begin{equation}\label{vardetermLast}
V(X_{N},R)=
 R^{d}\sum\limits_{w'\in \Lambda^{*}\cap\mathcal{D}_{N}}
\mathop{\sum}\limits_{w\in \Lambda^{*},  w\notin\mathcal{D}_{N},  \atop |w|<2c_{0} N^{\frac{1}{d}}}
|w-w'|^{-d} \Big(J_{\frac{d}{2}}(2\pi|w-w'|R) \Big)^{2}+\mathcal{O}\Big(R^{d-1}N^{1-\frac{1}{d}}\Big).
\end{equation}

We will use the representation of number variance $V(X_{N},R)$ from the  formula \eqref{vardetermLast} to prove the following Theorem.

\begin{theorem}\label{theorem_determin}
The determinantal point process  is hyperuniform for large and small balls.
In the threshold regime the weaker property 
\begin{equation}\label{determinantalSmall}
\lim\limits_{N\rightarrow\infty}V(X_{N},tN^{-\frac{1}{d}})=\mathcal{O}\big( t^{d-1}\ln t \big)
\end{equation} 
holds.
\end{theorem}
\begin{proof}[Proof of Theorem  \ref{theorem_determin}]

{\bf (i) Large ball regime:}

By the estimate \eqref{asympLarge} one can easily verify  that for $R\in(0,\frac{1}{2}\mathrm{diam}\Omega_{\Lambda})$
\begin{align}\label{vardeterm33}
&R^{d}\mathop{\sum}\limits_{w'\in \Lambda^{*}, \atop  c_{0}N^{\frac{1}{d}}-1\leq |w'|< c_{0}N^{\frac{1}{d}}+c_{1}+1}
\mathop{\sum}\limits_{w\in \Lambda^{*},  w\notin\mathcal{D}_{N},  \atop |w|<2c_{0}N^{\frac{1}{d}}}
|w-w'|^{-d} \Big(J_{\frac{d}{2}}(2\pi|w-w'|R) \Big)^{2} \notag \\
\ll &
R^{d-1}\mathop{\sum}\limits_{w'\in \Lambda^{*}, \atop  c_{0}N^{\frac{1}{d}}-1\leq |w'|< c_{0}N^{\frac{1}{d}}+c_{1}+1}
\mathop{\sum}\limits_{w\in \Lambda^{*},  w\notin\mathcal{D}_{N},  \atop |w|<2c_{0} N^{\frac{1}{d}}}
|w-w'|^{-d-1}
\ll R^{d-1}N^{1-\frac{1}{d}}.
\end{align} 

Thus, \eqref{vardetermLast}, and \eqref{vardeterm33} and Lemma~\ref{lemma_sum} imply 
\begin{align}\label{vardeterm4}
V(X_{N},R)
&\ll
R^{d-1}\mathop{\sum}\limits_{w'\in \Lambda^{*}, \atop   |w'|< c_{0}N^{\frac{1}{d}}-1}
\mathop{\sum}\limits_{w\in \Lambda^{*},    \atop c_{0}N^{\frac{1}{d}}+c_{1}+1\leq|w|< 2c_{0}N^{\frac{1}{d}}}
|w-w'|^{-d-1} +
R^{d-1}N^{1-\frac{1}{d}} \notag \\
&\ll R^{d-1}N^{1-\frac{1}{d}}\ln N.
\end{align}

So, from  \eqref{vardeterm4} we have that
\begin{equation}
V(X_{N},R)=
\mathcal{O}(N^{1-\frac{1}{d}}\ln N)=o(N), \ \ \ \mathrm{as} \ N\rightarrow\infty,
\end{equation} 
which implies, that the determinantal point process  is hyperuniform for large balls.

{\bf (ii) Small ball regime:}
Let $\lim\limits_{N\rightarrow\infty}R_{N}=0$. Then by \eqref{vardetermLast}
\begin{align}
& \ \  V(X_{N},R)=
 R_{N}^{d}\sum\limits_{w'\in \Lambda^{*}\cap\mathcal{D}_{N}}
\mathop{\sum}\limits_{w\in \Lambda^{*},  w\notin\mathcal{D}_{N},  \atop 0<|w-w'|<\frac{1}{R_{N}}}
|w-w'|^{-d} \Big(J_{\frac{d}{2}}(2\pi|w-w'|R_{N}) \Big)^{2}\label{vardetermSmall} \\
&
+
\! R_{N}^{d}\!\!\sum\limits_{w'\in \Lambda^{*}\cap\mathcal{D}_{N}} \!\!\!\!\!\!
\mathop{\sum}\limits_{w\in \Lambda^{*},  w\notin\mathcal{D}_{N},  \atop \frac{1}{R_{N}} \leq|w-w'|<3c_{0} N^{\frac{1}{d}}} \!\!\!\!\!\!\!\!\!
|w-w'|^{-d} \Big(J_{\frac{d}{2}}(2\pi|w-w'|R_{N}) \Big)^{2} 
\!+\!\mathcal{O}\Big(R_{N}^{d-1}N^{1-\frac{1}{d}}\Big). \notag
\end{align} 

Using \eqref{asympSmall}, we get
\begin{align}
&
 R_{N}^{d}\sum\limits_{w'\in \Lambda^{*}\cap\mathcal{D}_{N}}
\mathop{\sum}\limits_{w\in \Lambda^{*},  w\notin\mathcal{D}_{N},  \atop 0<|w-w'|<\frac{1}{R_{N}}}
|w-w'|^{-d} \Big(J_{\frac{d}{2}}(2\pi|w-w'|R_{N}) \Big)^{2} \notag \\
 \ll & R_{N}^{2d}\!\!\sum\limits_{w'\in \Lambda^{*}\cap\mathcal{D}_{N}}
\mathop{\sum}\limits_{w\in \Lambda^{*},  w\notin\mathcal{D}_{N},  \atop 0<|w-w'|<\frac{1}{R_{N}}}\!\!\!\!\!
1
\ll R_{N}^{2d} \sum\limits_{m=[c_{0}N^{\frac{1}{d}}-\frac{1}{R_{N}}]}^{[c_{0}N^{\frac{1}{d}}]+1}\sum\limits_{m\leq |w'|<m+1}\sum\limits_{k=1}^{[\frac{1}{R_{N}}]+1}\!\!\sum\limits_{k\leq|w-w'|<k+1} \!\!\!\! \!\!1
 \label{vardetermSmallFirst} \\
 \ll & R_{N}^{2d} \sum\limits_{m=[c_{0}N^{\frac{1}{d}}-\frac{1}{R_{N}}]}^{[c_{0}N^{\frac{1}{d}}]+1}m^{d-1}\sum\limits_{k=1}^{[\frac{1}{R_{N}}]+1}k^{d-1}\ll
R_{N}^{d-1}N^{1-\frac{1}{d}}. \notag
\end{align} 

Applying \eqref{asympLarge}, Lemma~\ref{lemma_sum} and \eqref{relat2}, we have
\begin{align*}
&
 R_{N}^{d}\sum\limits_{w'\in \Lambda^{*}\cap\mathcal{D}_{N}}
\mathop{\sum}\limits_{w\in \Lambda^{*},  w\notin\mathcal{D}_{N},  \atop \frac{1}{R_{N}} \leq|w-w'|<3c_{0} N^{\frac{1}{d}}}
|w-w'|^{-d} \Big(J_{\frac{d}{2}}(2\pi|w-w'|R_{N}) \Big)^{2}
\notag \\
 \ll &
 R_{N}^{d-1}\sum\limits_{w'\in \Lambda^{*}\cap\mathcal{D}_{N}}
\mathop{\sum}\limits_{w\in \Lambda^{*},  w\notin\mathcal{D}_{N},  \atop \frac{1}{R_{N}} \leq|w-w'|<3c_{0} N^{\frac{1}{d}}}
|w-w'|^{-d-1}
 \notag \\
  \ll &
 R_{N}^{d-1}\mathop{\sum}\limits_{w'\in \Lambda^{*},    \atop \frac{1}{2}c_{o}N^{\frac{1}{d}}|w'|\leq c_{o}N^{\frac{1}{d}}+1} 
 \mathop{\sum}\limits_{w\in \Lambda^{*},  w\notin\mathcal{D}_{N},  \atop \frac{1}{R_{N}} \leq|w-w'|<3c_{0} N^{\frac{1}{d}}}
|w-w'|^{-d-1}+R_{N}^{d-1}  N^{1-\frac{1}{d}}
\end{align*}
\begin{align}
 & \ll 
 R_{N}^{d-1}\!\! \sum\limits_{k=1}^{[\frac{1}{2}c_{o}N^{\frac{1}{d}}]} \!\!\!\sum\limits_{[c_{o}N^{\frac{1}{d}}]-k\leq |w'|< [c_{o}N^{\frac{1}{d}}]-k+1}
  \mathop{\sum}\limits_{w\in \Lambda^{*},  w\notin\mathcal{D}_{N},  \atop |w-w'|\geq\frac{1}{R_{N}}+k}
|w-w'|^{-d-1}+R_{N}^{d-1}  N^{1-\frac{1}{d}}
\label{vardetermSmallSecond} \\
 & \ll 
 R_{N}^{d-1} N^{1-\frac{1}{d}}
 \sum\limits_{k=1}^{[\frac{1}{2}c_{o}N^{\frac{1}{d}}]} \Big(k+\frac{1}{R_{N}}\Big)^{-1}\ll R_{N}^{d-1}  N^{1-\frac{1}{d}}\ln (N^{\frac{1}{d}}R_{N}). \notag
\end{align} 

Combining \eqref{vardetermSmall}, \eqref{vardetermSmallFirst} and \eqref{vardetermSmallSecond}, we have
\begin{align}\label{VarianceHyperSmall}
&V(X_{N},R_N)
=\mathcal{O}\Big( R_{N}^{d-1}  N^{1-\frac{1}{d}}\ln (N^{\frac{1}{d}}R_{N}) \Big)
=\mathcal{O}\Big(  \frac{\ln (N^{\frac{1}{d}}R_{N})}{N^{\frac{1}{d}}R_{N}}  N\mathrm{Vol}(B(\cdot,R_{N})) \Big)
\notag \\
&=o\Big(N\mathrm{Vol}(B(\cdot,R_{N})) \Big) ,
\end{align} 
where we have used that $N^{\frac{1}{d}}R_{N}\rightarrow\infty$ as $N\rightarrow\infty$. This proves hyperuniformity for small balls.

{\bf (iii) Threshold regime}
Let $R_{N}=tN^{-\frac{1}{d}}$, $R\in(0,\frac{1}{2}\mathrm{diam}\Omega_{\Lambda})$ and $t>0$.
From \eqref{vardetermSmall}--\eqref{vardetermSmallSecond}, we have
\begin{equation}\label{form33}
 V(X_{N},R)
 \ll (tN^{-\frac{1}{d}})^{d-1}N^{1-\frac{1}{d}}\ln t=t^{d-1}\ln t \ \ \mathrm{as}   \ \ N\rightarrow\infty.
\end{equation}

Since $t>0$ was arbitrary, then
\begin{equation}
 \limsup\limits_{N\rightarrow\infty}V(X_{N},tN^{-\frac{1}{d}})
=\mathcal{O}(t^{d-1}\ln t)  \ \ \mathrm{as} \ \ t\rightarrow\infty.
\end{equation}
Theorem~\ref{theorem_determin} is proved.

\end{proof}

%\section{The geodesic distance energy integral. }\label{s.pgeod}

\section*{Acknowledgments} I would like to express my  huge gratitude to Peter Grabner, who posed this problem, gave many valuable advices  and ideas and  also for  his support and fruitful discussions. Also I would like to thank to Dmitriy Bilyk for his valuable comments and help in proof of Lemma~\ref{lemma_sum}.

%\bibliography{FlatTorusHyper}

%\begin{thebibliography}{999999999}

%\bibitem{PftB} M. Aigner,  G. Ziegler, Proofs from The Book. Fifth edition.  Springer--Verlag, Berlin, 2014.

%\bibitem{Beck1} J. Beck,  Some upper bounds in the theory of irregularities of distribution. Acta Arith. 43, no. 2, 115--130, 1984.

%\bibitem{Beck2} J. Beck,  Sums of distances between points on a sphere--an application of the theory of irregularities of distribution to discrete geometry. Mathematika 31, no. 1, 33--41, 1984.

%\end{thebibliography}

\end{document}